\documentclass[12pt]{article}
\usepackage[margin=4.46cm]{geometry}
\usepackage{graphicx}
\usepackage{epsfig}
\usepackage{caption}
\usepackage{amsmath,amsfonts,amssymb,amsthm}
\usepackage[T1]{fontenc}
\usepackage{cases}
\usepackage{enumerate}
\usepackage{graphicx,color}
\usepackage{anyfontsize}
\usepackage{float}
\newcommand{\bsig}{\boldsymbol{\sigma}}
\newcommand{\commentout}[1]{}

\newcommand{\no}{\nonumber}
\newcommand{\be}{\begin{equation}}
\newcommand{\ee}{\end{equation}}
\newcommand{\ba}{\begin{eqnarray}}
\newcommand{\ea}{\end{eqnarray}}
\newcommand{\bi}{\begin{itemize}}
\newcommand{\ei}{\end{itemize}}

\newcommand{\eps}{\mbox{$\epsilon$}}

\newcommand{\real}{\mathbb{R}}

\newtheorem{theo}{Theorem}[section]

\newtheorem{prop}{Proposition}[section]

\newtheorem{cor}{Corollary}[section]

\allowdisplaybreaks
\begin{document}
\title{Residual Diffusivity in Elephant Random Walk Models with Stops}
\author{Jiancheng Lyu, \ Jack Xin, \ Yifeng Yu \thanks{Department of Mathematics, University of California at Irvine, Irvine, CA 92697.
Email: (jianchel,jack.xin,yifengy)@uci.edu. The work was partly supported by 
NSF grants DMS-1211179 (JX), DMS-1522383 (JX), DMS-0901460 (YY), and CAREER Award DMS-1151919 (YY).}}
\date{}
\maketitle

\begin{abstract}
We study the enhanced diffusivity in the so called elephant random walk model with stops (ERWS) by including symmetric random walk steps at small probability $\epsilon$. At any $\epsilon > 0$, the large time behavior transitions from sub-diffusive at $\epsilon = 0$ to diffusive in a wedge shaped parameter regime where the diffusivity is strictly above that in the un-perturbed ERWS model in the $\epsilon \downarrow 0$ limit. The perturbed ERWS model is shown to be solvable with the first two moments and their asymptotics calculated exactly in both one and two space dimensions. The model provides a discrete analytical setting of the residual diffusion phenomenon known for the passive scalar transport in chaotic flows (e.g. generated by time periodic cellular flows and statistically sub-diffusive) as molecular diffusivity tends to zero.     
\end{abstract}
\vspace{.2 in}

\hspace{.12 in} {\bf AMS Subject Classification:} 60G50, 60H30, 58J37.
\bigskip

\hspace{.12 in} {\bf Key Words:} Elephant random walk with stops, 

\hspace{.12 in} sub-diffusion, moment analysis, residual diffusivity.

\thispagestyle{empty}
\newpage

\section{Introduction}
\setcounter{equation}{0}
\setcounter{page}{1}
Residual diffusion is a remarkable phenomenon arising in large scale 
fluid transport from chaotic flows \cite{BCVV95,Mur17,LXY17}. 
It refers to the positive macroscopic effective diffusivity ($D^E$) as the microscopic 
molecular diffusivity ($D_0$) approaches zero, in the broader context of flow enhanced 
turbulent diffusion that has been studied for nearly a century \cite{T21,MK99}. 
An example of a chaotic smooth flow is the particle trajectories 
of the time periodic cellular flow ($X=(x,y) \in \real^2$):
\be\label{tdcell}
\boldsymbol{v}(X,t)=(\cos (y),\cos (x) ) + \theta\; \cos (t)\;(\sin (y),\sin (x)),
\quad
\theta \in (0,1].
\ee 
The first term of (\ref{tdcell}) is a steady cellular flow consisting of a periodic array of vortices, and the second term is a time periodic perturbation that introduces an increasing amount of disorder in the flow trajectories as $\theta$ becomes larger. At $\theta =1$, the flow is fully mixing, and empirically sub-diffusive \cite{ZCX15}. The flow (\ref{tdcell}) is a simple model of chaotic advection in Rayleigh-B\'enard experiment \cite{CW91}. The motion of a diffusing particle in the flow (\ref{tdcell}) satisfies the stochastic differential equation (SDE):
\be \label{sde1}
dX_t = \boldsymbol{v}(X_t,t)\, dt + \sqrt{2\,D_0}\, dW_t,\;\; X(0)=(x_0,y_0) \in \real^2, 
\ee
where $D_0> 0$ is molecular diffusivity, $W_t$ is the standard 2-dimensional Wiener process. The mean square displacement in the unit direction $e$ at large times is given by \cite{BLP2011}:
\be \label{sde2}
\lim_{t \uparrow +\infty} \, E(|(X(t) - X(0))\cdot e |^{2})/t = D^E,  
\ee
where $D^E = D^E(D_0,e,\theta) > D_0$ is the effective diffusivity. Numerical simulations \cite{BCVV95,Mur17,LXY17} based 
on the associated Fokker-Planck equations suggest that at $e=(1,0)$, $\theta = 1$, $D^E = O(1)$ as $D_0 \downarrow 0$, 
the {\it residual diffusion} occurs. In fact, $D^E = O(1)$ for $e=(0,1)$ and a range of values in $\theta \in (0,1)$ as 
well \cite{LXY17}. In contrast, at $\theta=0$, $D^{E}=O(\sqrt{D_0})$ as $D_0 \downarrow 0$, see  
\cite{FP94,H03,NPR05} for various proofs and generalized steady cellular flows. 
\medskip

Currently, the mathematical theory of residual diffusion remains elusive. 
In this paper, we analyze the residual diffusion phenomenon in a random walk model which is 
solvable in the sense of moments and has certain statistical features of the SDE model (\ref{sde2}). 
The baseline random walk model is the so called elephant random walk model with stops (ERWS) \cite{NUK10} which 
is non-Markovian and exhibits sub-diffusive, diffusive and super-diffusive regimes. For a 
review on various stochastic models of animal movement (including SDE and random walk models), 
momory effects and anomalous diffusion, see \cite{Sm10}. The sub-diffusive regime is absent in 
the earlier version of the ERW model without stops \cite{ST04}.
Stops in random walk models are often interpreted as occasional periods of 
rest during an animal's movement \cite{TPN16}.
Recall that the chaotic system from (\ref{sde1}) is sub-diffusive \cite{ZCX15} at $D_0=0$ and 
transitions to diffusive with residual diffusion at $D_0 >0$. To mimic this in the ERWS model, 
we add a small probability of symmetric random walk in the sub-diffusive regime and 
examine the large time behavior of the mean square displacement. 
Interestingly, the sub-diffusive regime also transitions into diffusive regime 
and a wedge shaped parameter region appears where the diffusivity is {\rm strictly above} that of the 
baseline ERWS model in the zero probability limit of the 
symmetric random walk (analogue of the zero molecular diffusivity limit). In the context of 
animal dispersal in ecology, the emergence of residual diffusion indicates that the large time statistical 
behavior of the movement can pick up positive normal diffusivity 
when the animal's rest pattern is slightly disturbed consistently in time.
We also extend our analysis to a two dimensional ERWS model (see \cite{CVS13} for 
a related solvable model). It is our hope that more diverse mathematical models of 
residual diffusivity can be developed and analyzed towards 
gaining understanding of the SDE residual diffusivity problem (\ref{sde1})-(\ref{sde2}) in the future. 
\medskip

The rest of the paper is organized as follows. 
In section 2, we present our perturbed model which is ERWS with a small probability ($\epsilon$) of symmetric 
random walk, analyze the first two moments and derive the large time asymptotics of the second moment. 
The residual diffusive behavior follows. In section 3, we generalize our results to a two dimensional perturbed ERWS model. Conclusions are in section 4.  

\medskip

\section{Perturbed ERWS and Moment Analysis}
\setcounter{equation}{0}
In this section, we show the perturbed ERWS model in one dimension and the analysis of the first two moments leading up to residual diffusivity.
\subsection{Perturbed ERWS}
Consider a random walker on a one-dimensional lattice with unit distance between adjacent lattice sites. Denote the position of the walker at time $t$ by $X_t$. Time is discrete ($t = 0, 1, 2, \dots$) and the walker starts at the origin, $X_0 = 0$. At each time step, $t \rightarrow t+1$, 
\begin{align*}
X_{t+1} = X_t+\sigma_{t+1},
\end{align*}
where $\sigma_{t+1} \in \left\{-1, 0, 1\right\}$ is a random number depending on $\left\{\sigma_t\right\} = \left(\sigma_1, \dots, \sigma_t\right)$ as follows. Let $p, q, r, \epsilon \in \left(0, 1\right)$ and $p+q+r=1$. The process is started at time $t=0$ by allowing the walker to move to the right with probability $s$ and to the left with probability $1-s$, $s\in \left(0, 1\right)$. For $t \geq 1$, a random previous time $k \in \left\{1,\dots, t\right\}$ is chosen with uniform probability.
\begin{enumerate}[(i)]
\item If $\sigma_k = \pm1$,
\begin{align*}
&P\left(\sigma_{t+1}=\sigma_k\right) = p,\; P\left(\sigma_{t+1}=-\sigma_k\right) = q,\\
&P\left(\sigma_{t+1}=0\right) = r.
\end{align*}
\item If $\sigma_k = 0$,
\begin{align*}
&P\left(\sigma_{t+1}=1\right) = P\left(\sigma_{t+1}=-1\right) = \epsilon/2,\\
&P\left(\sigma_{t+1}=0\right) = 1-\epsilon.
\end{align*}
\end{enumerate}
When $\eps =0$, the above model reduces to the ERWS model of \cite{NUK10}.

\subsection{Moment Analysis}
We calculate the first and second moments of $X_t$ below. 

\subsubsection{First moment $\langle X_t\rangle$}
At $t = 0$, it follows from the initial condition of the model for $\sigma = \pm1$ that
\begin{align*}
P\left(\sigma_1=\sigma\right) = \dfrac{1}{2}\left[1+\left(2s-1\right)\sigma\right].
\end{align*}
Let $\gamma = p-q$, for $t \geq 1$, it follows from the probabilistic structure of the model and $\sigma_k \in \left\{1, -1, 0\right\}$ that
\begin{align*}
P\left(\left.\sigma_{t+1}=1\right|\left\{\sigma_t\right\}\right) &= \dfrac{1}{t}\sum_{k=1}^t\left[\sigma_k^2\left(1\!+\!\sigma_k\right)\dfrac{p}{2}+\sigma_k^2\left(1\!-\!\sigma_k\right)\dfrac{q}{2}+\left(1\!-\!\sigma_k^2\right)\dfrac{\epsilon}{2}\right]\\
&= \dfrac{1}{t}\sum_{k=1}^t\left[\sigma_k^2\dfrac{1-r}{2}+\sigma_k\dfrac{\gamma}{2}+\left(1-\sigma_k^2\right)\dfrac{\epsilon}{2}\right]\\
&= \dfrac{1}{2t}\sum_{k=1}^t\left[\sigma_k^2\left(1-\epsilon-r\right)+\sigma_k\gamma\right]+\dfrac{\epsilon}{2},
\end{align*}
\begin{align*}
P\left(\left.\sigma_{t+1}=-1\right|\left\{\sigma_t\right\}\right) &= \dfrac{1}{t}\sum_{k=1}^t\left[\sigma_k^2\left(1\!-\!\sigma_k\right)\dfrac{p}{2}+\sigma_k^2\left(1\!+\!\sigma_k\right)\dfrac{q}{2}+\left(1\!-\!\sigma_k^2\right)\dfrac{\epsilon}{2}\right]\\
&= \dfrac{1}{t}\sum_{k=1}^t\left[\sigma_k^2\dfrac{1-r}{2}-\sigma_k\dfrac{\gamma}{2}+\left(1-\sigma_k^2\right)\dfrac{\epsilon}{2}\right]\\
&= \dfrac{1}{2t}\sum_{k=1}^t\left[\sigma_k^2\left(1-\epsilon-r\right)-\sigma_k\gamma\right]+\dfrac{\epsilon}{2},
\end{align*}
\begin{align*}
P\left(\left.\sigma_{t+1}=0\right|\left\{\sigma_t\right\}\right) &= \dfrac{1}{t}\sum_{k=1}^t\left[\sigma_k^2r+\left(1-\sigma_k^2\right)\left(1-\epsilon\right)\right]\\
&= \dfrac{1}{t}\sum_{k=1}^t\left[-\sigma_k^2\left(1-\epsilon-r\right)\right]+1-\epsilon\\
&= \dfrac{1}{2t}\sum_{k=1}^t\left[-2\sigma_k^2\left(1-\epsilon-r\right)\right]+1-\epsilon.
\end{align*}
Therefore, for $\sigma = \pm1, 0$,
\begin{align*}
P\left(\left.\sigma_{t+1}=\sigma\right|\left\{\sigma_t\right\}\right) = &\dfrac{1}{2t}\sum_{k=1}^t\left[\sigma_k^2\left(3\sigma^2-2\right)\left(1-\epsilon-r\right)+\sigma\sigma_k\gamma\right]\\
&+\dfrac{\sigma^2}{2}\epsilon+\left(1-\sigma^2\right)\left(1-\epsilon\right).
\end{align*}

The conditional mean value of $\sigma_{t+1}$ for $t \geq 1$ is
\begin{align*}
\langle\left.\sigma_{t+1}\right|\left\{\sigma_t\right\}\rangle = &\sum_{\sigma=\pm1,0}\sigma P\left(\left.\sigma_{t+1}=\sigma\right|\left\{\sigma_t\right\}\right)\\
= &\sum_{\sigma=\pm1}\sigma\left\{\dfrac{1}{2t}\sum_{k=1}^t\left[\sigma_k^2\left(3\sigma^2-2\right)\left(1-\epsilon-r\right)+\sigma\sigma_k\gamma\right]\right. \\
&\left.+\dfrac{\sigma^2}{2}\epsilon+\left(1-\sigma^2\right)\left(1-\epsilon\right)\right\}\\
= &\sum_{\sigma=\pm1}\sigma\left\{\dfrac{1}{2t}\sum_{k=1}^t\left[\sigma_k^2\left(1-\epsilon-r\right)+\sigma\sigma_k\gamma\right]+\dfrac{\epsilon}{2}\right\}\\
= &\sum_{\sigma=\pm1}\dfrac{1}{2t}\sum_{k=1}^t\sigma^2\sigma_k\gamma,
\end{align*}
hence,
\begin{align}\label{sigma1}
\langle\left.\sigma_{t+1}\right|\left\{\sigma_t\right\}\rangle = \dfrac{\gamma}{t}X_t.
\end{align}
It follows that
\begin{align*}
\langle\sigma_{t+1}\rangle = \dfrac{\gamma}{t}\langle X_t\rangle,
\end{align*}
therefore
\begin{align*}
\langle X_{t+1}\rangle = \left(1+\dfrac{\gamma}{t}\right)\langle X_t\rangle.
\end{align*}
By the initial condition $\langle X_1\rangle = 2s-1$,
\begin{align*}
\langle X_t\rangle = \left(2s-1\right)\dfrac{\Gamma\left(t+\gamma\right)}{\Gamma\left(1+\gamma\right)\Gamma\left(t\right)}.
\end{align*}
Since $\displaystyle\lim_{t\rightarrow\infty}\dfrac{\Gamma\left(t+\alpha\right)}{\Gamma\left(t\right)t^\alpha} = 1$, $\forall \alpha$,
\begin{align*}
\langle X_t\rangle \sim \dfrac{2s-1}{\Gamma\left(1+\gamma\right)} t^\gamma, \quad t \rightarrow \infty.
\end{align*}
For simplicity, we shall take $s=1/2$ below, and so $\langle X_t\rangle = 0$, the mean square displacement agrees with the second moment.

\subsubsection{Second moment $\langle X_t^2\rangle$}

The conditional mean value of $\sigma_{t+1}^2$ for $t \geq 1$ is
\begin{align*}
\langle\left.\sigma_{t+1}^2\right|\left\{\sigma_t\right\}\rangle = &\sum_{\sigma=\pm1,0}\sigma^2P\left(\left.\sigma_{t+1}=\sigma\right|\left\{\sigma_t\right\}\right)\\
= &\sum_{\sigma=\pm1}\sigma^2\left\{\dfrac{1}{2t}\sum_{k=1}^t\left[\sigma_k^2\left(3\sigma^2-2\right)\left(1-\epsilon-r\right)+\sigma\sigma_k\gamma\right]\right.\\
&\left.+\dfrac{\sigma^2}{2}\epsilon+\left(1-\sigma^2\right)\left(1-\epsilon\right)\right\}\\
= &\sum_{\sigma=\pm1}\left\{\dfrac{1}{2t}\sum_{k=1}^t\left[\sigma_k^2\left(1-\epsilon-r\right)+\sigma\sigma_k\gamma\right]+\dfrac{\epsilon}{2}\right\}\\
= &\sum_{\sigma=\pm1}\left\{\dfrac{1}{2t}\sum_{k=1}^t\sigma_k^2\left(1-\epsilon-r\right)+\dfrac{\epsilon}{2}\right\}\\
= &\dfrac{1-\epsilon-r}{t}\sum_{k=1}^t\sigma_k^2+\epsilon.
\end{align*}
It follows that
\begin{align*}
\langle\left.\sigma_{t+1}^2\right|\left\{\sigma_t\right\}\rangle &= \dfrac{1-\epsilon-r}{t}\sum_{k=1}^{t-1}\sigma_k^2+\dfrac{1-\epsilon-r}{t}\sigma_t^2+\epsilon\\
&= \dfrac{t-1}{t}\left(\dfrac{1\!-\!\epsilon\!-\!r}{t-1}\sum_{k=1}^{t-1}\sigma_k^2+\epsilon\right)-\dfrac{t-1}{t}\epsilon+\dfrac{1-\epsilon-r}{t}\sigma_t^2+\epsilon\\
&= \dfrac{t-1}{t}\langle\left.\sigma_t^2\right|\left\{\sigma_{t-1}\right\}\rangle+\dfrac{1-\epsilon-r}{t}\sigma_t^2+\dfrac{\epsilon}{t},
\end{align*}
so
\begin{align}\label{sigmarc}\begin{split}
\langle\sigma_1^2\rangle &= 1,\\
\langle\sigma_{t+1}^2\rangle &= \left(1-\dfrac{\epsilon+r}{t}\right)\langle\sigma_t^2\rangle+\dfrac{\epsilon}{t}.
\end{split}
\end{align}

Since
\begin{align*}
\langle\left.X_{t+1}^2\right|\left\{\sigma_t\right\}\rangle = X_t^2+2X_t\langle\left.\sigma_{t+1}\right|\left\{\sigma_t\right\}\rangle+\langle\left.\sigma_{t+1}^2\right|\left\{\sigma_t\right\}\rangle,
\end{align*}
by \eqref{sigma1},
\begin{align}\label{moment2rc}
\langle X_{t+1}^2\rangle = \left(1+\dfrac{2\gamma}{t}\right)\langle X_t^2\rangle+\langle\sigma_{t+1}^2\rangle.
\end{align}

To motivate the solution we shall present, let us consider the ODE analogue of the difference equations \eqref{sigmarc} and \eqref{moment2rc}.
\begin{equation}\label{ode}\begin{cases}
&x'+\dfrac{\epsilon+r}{t}x = \dfrac{\epsilon}{t},\\
&y'-\dfrac{2\gamma}{t}y = x.
\end{cases}
\end{equation}
The solution to \eqref{ode} is
\begin{align*}\begin{cases}
&x\left(t\right) = \dfrac{C}{t^{\epsilon+r}}+\dfrac{\epsilon}{\epsilon+r},\\
&y\left(t\right) = \dfrac{\epsilon}{\left(1-2\gamma\right)\left(\epsilon+r\right)}t+\dfrac{C}{1-\epsilon-r-2\gamma}t^{1-\epsilon-r}+Dt^{2\gamma},
\end{cases}
\end{align*}
if $\gamma \neq \dfrac{1}{2}$, and
\begin{align*}\begin{cases}
&x\left(t\right) = \dfrac{C}{t^{\epsilon+r}}+\dfrac{\epsilon}{\epsilon+r},\\
&y\left(t\right) = \dfrac{\epsilon}{\epsilon+r}t\ln t-\dfrac{C}{\epsilon+r}t^{1-\epsilon-r}+Dt,
\end{cases}
\end{align*}
if $\gamma = \dfrac{1}{2}$, where $C$ and $D$ are constants.

\begin{prop}\label{ppsigma}
The solution to \eqref{sigmarc} is
\begin{align}\label{sigmaform}
\langle\sigma_t^2\rangle = C\dfrac{\Gamma\left(t-\epsilon-r\right)}{\Gamma\left(t\right)}+\dfrac{\epsilon}{\epsilon+r},
\end{align}
where
\begin{align*}
C = \dfrac{r}{\left(\epsilon+r\right)\Gamma\left(1-\epsilon-r\right)}.
\end{align*}
\end{prop}
\begin{proof}
Clearly,
\begin{align*}
\dfrac{\epsilon}{\epsilon+r} &= \left(1-\dfrac{\epsilon+r}{t}\right)\dfrac{\epsilon}{\epsilon+r}+\dfrac{\epsilon}{t},\\
\dfrac{\Gamma\left(t+1-\epsilon-r\right)}{\Gamma\left(t+1\right)} &= \left(1-\dfrac{\epsilon+r}{t}\right)\dfrac{\Gamma\left(t-\epsilon-r\right)}{\Gamma\left(t\right)},
\end{align*}
so a general solution to the recurrence equation in \eqref{sigmarc} is given by \eqref{sigmaform}.
The initial condition $\langle\sigma_1^2\rangle = 1$ implies $C = \dfrac{r}{\left(\epsilon+r\right)\Gamma\left(1-\epsilon-r\right)}$.
\end{proof}

It follows from Proposition \ref{ppsigma} and \eqref{moment2rc} that
\begin{align}\begin{split}\label{moment2rc1}
\langle X_1^2\rangle &= 1,\\
\langle X_{t+1}^2\rangle &= \left(1+\dfrac{2\gamma}{t}\right)\langle X_t^2\rangle+C\dfrac{\Gamma\left(t+1-\epsilon-r\right)}{\Gamma\left(t+1\right)}+\dfrac{\epsilon}{\epsilon+r}.
\end{split}
\end{align}

\begin{theo}\label{thmmoment2}
\begin{enumerate}[(1)]
\item If $\gamma \neq \dfrac{1}{2}$, the solution to \eqref{moment2rc1} is
\begin{align}\label{moment2}
\langle X_t^2\rangle = \dfrac{\epsilon}{\left(1-2\gamma\right)\left(\epsilon+r\right)}t+\dfrac{C}{1\!-\!\epsilon\!-\!r\!-\!2\gamma}\dfrac{\Gamma\left(t\!+\!1\!-\!\epsilon\!-\!r\right)}{\Gamma\left(t\right)}+D\dfrac{\Gamma\left(t+2\gamma\right)}{\Gamma\left(t\right)},
\end{align}
where
\begin{align*}
D = -\dfrac{1}{\Gamma\left(2\gamma\right)}\left[\dfrac{\epsilon}{\left(\epsilon+r\right)\left(1-2\gamma\right)}+\dfrac{r}{\left(\epsilon+r\right)\left(1-\epsilon-r-2\gamma\right)}\right].
\end{align*}
\item If $\gamma = \dfrac{1}{2}$, the solution to \eqref{moment2rc1} is
\begin{align}\label{moment21}
\langle X_t^2\rangle = \dfrac{\epsilon}{\epsilon+r}t\sum_{k=1}^t\dfrac{1}{k}-\dfrac{C}{\epsilon+r}\dfrac{\Gamma\left(t+1-\epsilon-r\right)}{\Gamma\left(t\right)}+Dt,
\end{align}
where
\begin{align*}
D = \dfrac{\epsilon}{\left(\epsilon+r\right)^2}-1.
\end{align*}
\end{enumerate}
\end{theo}

\begin{proof}
Motivated by the ODE solution, we check the formula of the solution to \eqref{moment2rc1}.

If $\gamma \neq \dfrac{1}{2}$, by the identity $\Gamma\left(x+1\right) = x\Gamma\left(x\right)$,
\begin{align}
\dfrac{\epsilon}{\left(1-2\gamma\right)\left(\epsilon+r\right)}\left(t+1\right) = &\left(1+\dfrac{2\gamma}{t}\right)\dfrac{\epsilon}{\left(1-2\gamma\right)\left(\epsilon+r\right)}t+\dfrac{\epsilon}{\epsilon+r},\\
\nonumber\dfrac{C}{1\!-\!\epsilon\!-\!r\!-\!2\gamma}\dfrac{\Gamma\left(t\!+\!2\!-\!\epsilon\!-\!r\right)}{\Gamma\left(t+1\right)} = &\left(1+\dfrac{2\gamma}{t}\right)\dfrac{C}{1\!-\!\epsilon\!-\!r\!-\!2\gamma}\dfrac{\Gamma\left(t\!+\!1\!-\!\epsilon\!-\!r\right)}{\Gamma\left(t\right)}\\
\label{recur2}&+C\dfrac{\Gamma\left(t+1-\epsilon-r\right)}{\Gamma\left(t+1\right)},\\
\label{recur3}\dfrac{\Gamma\left(t+1+2\gamma\right)}{\Gamma\left(t+1\right)} = &\left(1+\dfrac{2\gamma}{t}\right)\dfrac{\Gamma\left(t+2\gamma\right)}{\Gamma\left(t\right)}.
\end{align}
Hence a general solution to the recurrence equation in \eqref{moment2rc1} is given by \eqref{moment2} for some constant $D$. Then $\langle X_1^2\rangle = 1$ and $C = \dfrac{r}{\left(\epsilon\!+\!r\right)\Gamma\left(1\!-\!\epsilon\!-\!r\right)}$ imply
\begin{align*}
\dfrac{\epsilon}{\left(1\!-\!2\gamma\right)\left(\epsilon\!+\!r\right)}+\dfrac{r\Gamma\left(2\!-\!\epsilon\!-\!r\right)}{\left(\epsilon\!+\!r\right)\left(1\!-\!\epsilon\!-\!r\!-\!2\gamma\right)\Gamma\left(1\!-\!\epsilon\!-\!r\right)}+D\Gamma\left(1\!+\!2\gamma\right) = 1,
\end{align*}
so
\begin{align*}
D = -\dfrac{1}{\Gamma\left(2\gamma\right)}\left[\dfrac{\epsilon}{\left(\epsilon+r\right)\left(1-2\gamma\right)}+\dfrac{r}{\left(\epsilon+r\right)\left(1-\epsilon-r-2\gamma\right)}\right].
\end{align*}

If $\gamma = \dfrac{1}{2}$, \eqref{recur2} and \eqref{recur3} still hold,
\begin{align*}
-\dfrac{C}{\epsilon+r}\dfrac{\Gamma\left(t+2-\epsilon-r\right)}{\Gamma\left(t+1\right)} = &\left(1+\dfrac{1}{t}\right)\left(-\dfrac{C}{\epsilon+r}\dfrac{\Gamma\left(t+1-\epsilon-r\right)}{\Gamma\left(t\right)}\right)\\
&+C\dfrac{\Gamma\left(t+1-\epsilon-r\right)}{\Gamma\left(t+1\right)},\\
t+1 = &\left(1+\dfrac{1}{t}\right)t.
\end{align*}
For the recurrence relation
\begin{align*}
a_{t+1} = \left(1+\dfrac{1}{t}\right)a_t+\dfrac{\epsilon}{\epsilon+r},
\end{align*}
suppose $a_t = tb_t$, then
\begin{align*}
b_{t+1} = b_t + \dfrac{\epsilon}{\epsilon+r}\dfrac{1}{t+1},
\end{align*}
so for $t \geq 1$,
\begin{align*}
b_t = b_0+\dfrac{\epsilon}{\epsilon+r}\sum_{k=1}^t\dfrac{1}{k}.
\end{align*}
Set $b_0 = 0$, then
\begin{align*}
a_t = \dfrac{\epsilon}{\epsilon+r}t\sum_{k=1}^t\dfrac{1}{k}.
\end{align*}
Hence a general solution to the recurrence equation in \eqref{moment2rc1} in this case is \eqref{moment21}. Similarly, the initial condition gives
\begin{align*}
D = \dfrac{\epsilon}{\left(\epsilon+r\right)^2}-1.
\end{align*}
\end{proof}

The corollary below follows from \eqref{moment2} and \eqref{moment21}.
\begin{cor}
\begin{enumerate}[(1)]
\item If $\gamma \neq \dfrac{1}{2}$,
\begin{align*}
\langle X_t^2\rangle \sim \dfrac{\epsilon}{\left(1-2\gamma\right)\left(\epsilon+r\right)}t+\dfrac{C}{1-\epsilon-r-2\gamma}t^{1-\epsilon-r}+Dt^{2\gamma}, \quad t \rightarrow \infty.
\end{align*}
\item If $\gamma = \dfrac{1}{2}$,
\begin{align*}
\langle X_t^2\rangle \sim \dfrac{\epsilon}{\epsilon+r}t\ln t-\dfrac{C}{\epsilon+r}t^{1-\epsilon-r}+Dt, \quad t \rightarrow \infty.
\end{align*}
\end{enumerate}
\end{cor}

\subsection{Residual Diffusivity}
The occurrence of residual diffusivity relies on the choice of $\gamma $ as a function of $\epsilon$. 
To this end, we show three cases: 1) case 1 only recovers the un-perturbed diffusivity, 2) case 2 reveals the residual diffusivity exceeding the un-perturbed diffusivity in the limit of $\epsilon \downarrow 0$, 
3) case 3 results in residual super-diffusivity. The cases 2 and 3 are illustrated in Fig. 1. As $\epsilon \rightarrow 0$, the parameter region of the residual diffusion shrinks towards $\gamma = \dfrac{1}{2}$ while the enhanced diffusivity remains strictly above the un-perturbed diffusivity.

\subsubsection{Regular diffusivity: $\gamma = \dfrac{1-\epsilon}{2}$.}
Let $\gamma = \dfrac{1-\epsilon}{2}$, then $D = 0$ and
\begin{align*}
\langle X_t^2\rangle = \dfrac{1}{\left(\epsilon+r\right)}t-\dfrac{1}{\left(\epsilon+r\right)\Gamma\left(1-\epsilon-r\right)}\dfrac{\Gamma\left(t+1-\epsilon-r\right)}{\Gamma\left(t\right)},
\end{align*}
so
\begin{align*}
\langle X_t^2\rangle \sim \dfrac{1}{\left(\epsilon+r\right)}t-\dfrac{1}{\left(\epsilon+r\right)\Gamma\left(1-\epsilon-r\right)}t^{1-\epsilon-r}, \quad t \rightarrow \infty,
\end{align*}
and diffusivity equals $\dfrac{1}{\epsilon+r}$. 
\medskip

For fixed $r \in \left(0, \dfrac{1}{2}\right)$, let $\epsilon \in \left(0, 1\right)$, then
\begin{align*}
p = \dfrac{3-\epsilon-2r}{4}, \quad q = \dfrac{1+\epsilon-2r}{4}.
\end{align*}

Recall the second moment formula of \cite{NUK10} (equation (18)),
\ba
\langle X_t^2\rangle & = &
 \dfrac{1}{\left(2\gamma+r-1\right)\Gamma\left(t\right)}\left(\dfrac{\Gamma\left(t+2\gamma\right)}{\Gamma\left(2\gamma\right)}-\dfrac{\Gamma\left(1+t-r\right)}{\Gamma\left(1-r\right)}\right)
\no \\
& \sim & 
\dfrac{1}{\left(2\gamma+r-1\right)}\left(\dfrac{t^{2\gamma}}{\Gamma\left(2\gamma\right)}-\dfrac{t^{1-r}}{\Gamma\left(1-r\right)}\right),
\label{var0}
\ea
which is diffusive at $\gamma = 1/2$ with diffusivity $1/r$. 
\medskip 

We see that for $\gamma = (1-\epsilon)/2$, $\epsilon \in \left(0, 1\right)$ and the above $\left(p, q\right)$, 
the diffusivity of the perturbed ERW problem $1/\left(\epsilon+r\right)$ approaches $1/r$, the diffusivity of the un-perturbed model as $\epsilon \downarrow 0$. Hence no residual diffusivity exists.  

\subsubsection{Residual diffusivity: $\gamma = \dfrac{1-\epsilon r}{2}$.}

Let $\gamma = \dfrac{1-\epsilon r}{2}$, then
\begin{align*}
\langle X_t^2\rangle = &\dfrac{1}{r\left(\epsilon+r\right)}t-\dfrac{r\Gamma\left(t+1-\epsilon-r\right)}{\left(\epsilon+r\right)\left(\epsilon+r-\epsilon r\right)\Gamma\left(1-\epsilon-r\right)\Gamma\left(t\right)}\\
&-\dfrac{1}{\Gamma\left(1-\epsilon r\right)}\left[\dfrac{1}{r\left(\epsilon+r\right)}-\dfrac{r}{\left(\epsilon+r\right)\left(\epsilon+r-\epsilon r\right)}\right]\dfrac{\Gamma\left(t+1-\epsilon r\right)}{\Gamma\left(t\right)},
\end{align*}
and
\begin{align*}
\langle X_t^2\rangle \sim &\dfrac{1}{r\left(\epsilon+r\right)}t-\dfrac{r}{\left(\epsilon+r\right)\left(\epsilon+r-\epsilon r\right)\Gamma\left(1-\epsilon-r\right)}t^{1-\epsilon-r}\\
&-\dfrac{1}{\Gamma\left(1-\epsilon r\right)}\left[\dfrac{1}{r\left(\epsilon+r\right)}-\dfrac{r}{\left(\epsilon+r\right)\left(\epsilon+r-\epsilon r\right)}\right]t^{1-\epsilon r}, \quad t \rightarrow \infty.
\end{align*}
Hence
\begin{align*}
\lim_{t\rightarrow\infty}\dfrac{\langle X_t^2\rangle}{t} = \dfrac{1}{r\left(\epsilon+r\right)}.
\end{align*}

The diffusivity $\dfrac{1}{r\left(\epsilon+r\right)}$ can be much larger than $\dfrac{1}{r}$ in the un-perturbed model. 
In particular, given any $\delta > 0$, let $r_0 = \min\left\{\dfrac{1}{3}, \dfrac{1}{\delta}\right\}$, then for $r \in \left(0, r_0\right)$, $\epsilon \in \left(0, \dfrac{1}{6}\right)$,
\begin{align*}
\dfrac{1}{r\left(\epsilon+r\right)}-\dfrac{1}{r} = \dfrac{1}{r}\left(\dfrac{1}{\epsilon+r}\!-\!1\right) > \dfrac{1}{r_0}\left(\dfrac{1}{\frac{1}{6}+r_0}\!-\!1\right) \geq \delta\!\left(\dfrac{1}{\frac{1}{6}\!+\!\frac{1}{3}}\!-\!1\right) = \delta.
\end{align*}
The {\bf new diffusive region with residual diffusivity} is the {\bf wedge to the left of $\gamma = 1/2 $ covered by the dashed lines in Fig. 1}.

\subsubsection{Residual super-diffusivity: $\gamma = \dfrac{1+\epsilon r}{2}$}

If $\gamma = \dfrac{1+\epsilon r}{2}$, then
\begin{align*}
\langle X_t^2\rangle = &-\dfrac{1}{r\left(\epsilon+r\right)}t-\dfrac{r\Gamma\left(t+1-\epsilon-r\right)}{\left(\epsilon+r\right)\left(\epsilon+r+\epsilon r\right)\Gamma\left(1-\epsilon-r\right)\Gamma\left(t\right)}\\
&+\dfrac{1}{\Gamma\left(1+\epsilon r\right)}\left[\dfrac{1}{r\left(\epsilon+r\right)}+\dfrac{r}{\left(\epsilon+r\right)\left(\epsilon+r+\epsilon r\right)}\right]\dfrac{\Gamma\left(t+1+\epsilon r\right)}{\Gamma\left(t\right)},
\end{align*}
and
\begin{align*}
\langle X_t^2\rangle \sim &-\dfrac{1}{r\left(\epsilon+r\right)}t-\dfrac{r}{\left(\epsilon+r\right)\left(\epsilon+r+\epsilon r\right)\Gamma\left(1-\epsilon-r\right)}t^{1-\epsilon-r}\\
&+\dfrac{1}{\Gamma\left(1+\epsilon r\right)}\left[\dfrac{1}{r\left(\epsilon+r\right)}+\dfrac{r}{\left(\epsilon+r\right)\left(\epsilon+r+\epsilon r\right)}\right]t^{1+\epsilon r}, \quad t \rightarrow \infty.
\end{align*}
Thus at any $\epsilon > 0$, super-diffusion arises and
\begin{align*}
\lim_{t\rightarrow\infty}\dfrac{\langle X_t^2\rangle}{t^{1+\epsilon r}} = \dfrac{1}{\Gamma\left(1+\epsilon r\right)}\left[\dfrac{1}{r\left(\epsilon+r\right)}+\dfrac{r}{\left(\epsilon+r\right)\left(\epsilon+r+\epsilon r\right)}\right].
\end{align*}
As $\eps \downarrow 0$, the super-diffusivity tends to $r^{-2} + r^{-1} > r^{-1}$ the limiting 
super-diffusivity of the un-perturbed model as seen from (\ref{var0}).
The residual super-diffusive region is the wedge covered by lines to the right of $\gamma > 1/2$ in Fig. 1. 

\begin{figure}[H]
\centering
\begin{tabular}{c}
\includegraphics[width=.65\textwidth]{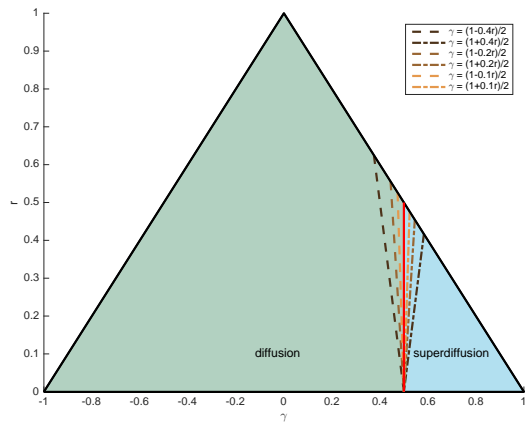}
\end{tabular}
\caption{Regions of residual diffusivity (wedge left of $\gamma=1/2$) and residual 
super-diffusivity (wedge right of $\gamma =1/2$) covered by the dashed lines at $\epsilon = 0.4, 0.2, 0.1$.}\label{fig1}
\end{figure}

\section{2D Perturbed ERWS Model}
In this section, we generalize our model to the two dimensional square lattice.
Let $\mathbf{i}, \mathbf{j}$ be the standard basis in 2D. Denote the position of the walker at time $t$ by $\boldsymbol{X}_t$,
\begin{align*}
\boldsymbol{X}_{t+1} = \boldsymbol{X}_t+\bsig_{t+1},
\end{align*}
where $\bsig_{t+1} \in \left\{\mathbf{i}, \mathbf{j}, -\mathbf{i}, -\mathbf{j}\right\}$. Let $s_i \in \left(0, 1\right)$, $i = 1, \dots, 4$ and the process is started by allowing the walker to move to the right, upward, to the left, downward with probability $s_1, \dots, s_4$. Let $p, q, q', r, \epsilon \in \left(0, 1\right)$ and $p+q+q'+r = 1$,
\begin{align*}
A = \begin{bmatrix}
0 & -1\\
1 & 0
\end{bmatrix}.
\end{align*}
For $t \geq 1$, a random $k \in \left\{1, \dots, t\right\}$ is chosen with uniform probability.
\begin{enumerate}[(i)]
\item If $\left|\bsig_k\right| = 1$,
\begin{align*}
&P\left(\bsig_{t+1}=\bsig_k\right) = p,\\
&P\left(\bsig_{t+1}=-\bsig_k\right) = q,\\
&P\left(\bsig_{t+1}=A\bsig_k\right) = p',\\
&P\left(\bsig_{t+1}=A^{-1}\bsig_k\right) = q',\\
&P\left(\bsig_{t+1}=\mathbf{0}\right) = r.
\end{align*}
\item If $\left|\bsig_k\right| = 0$,
\begin{align*}
&P\left(\bsig_{t+1}=\mathbf{i}\right) = P\left(\bsig_{t+1}=\mathbf{j}\right) = P\left(\bsig_{t+1}=-\mathbf{i}\right) = P\left(\bsig_{t+1}=-\mathbf{j}\right) = \epsilon/4,\\
&P\left(\bsig_{t+1}=\mathbf{0}\right) = 1-\epsilon.
\end{align*}
\end{enumerate}

Let $\gamma = p-q$, $\gamma' = p'-q'$, then for $t \geq 1$,
\begin{align*}
P\left(\left.\bsig_{t+1}\!=\!\bsig\right|\!\left\{\bsig_t\right\}\right) = &\dfrac{1}{t}\sum_{k=1}^t\left[\bsig_k\!\cdot\!\bsig\left(\bsig_k\!\cdot\!\bsig\!+\!1\right)\dfrac{p}{2}+\bsig_k\!\cdot\!\bsig\left(\bsig_k\!\cdot\!\bsig\!-\!1\right)\dfrac{q}{2}\right.\\
&+\bsig_k\!\cdot\!A\bsig\left(\bsig_k\!\cdot\!A\bsig\!+\!1\right)\dfrac{p'}{2}+\bsig_k\!\cdot\!A\bsig\left(\bsig_k\!\cdot\!A\bsig\!-\!1\right)\dfrac{q'}{2}\\
&\left.+\left(1-\left|\bsig_k\right|^2\right)\dfrac{\epsilon}{4}\right]\\
= &\dfrac{1}{2t}\sum_{k=1}^t\left[\bsig_k\!\cdot\!\bsig\gamma+\bsig_k\cdot A\bsig\gamma'+\left(\bsig_k\!\cdot\!\bsig\right)^2\left(p+q\right)\right.\\
&\left.+\left(\bsig_k\!\cdot\!A\bsig\right)^2\left(p'+q'\right)-\dfrac{1}{2}\left|\bsig_k\right|^2\epsilon\right]+\dfrac{\epsilon}{4},
\end{align*}
for $\left|\bsig\right| = 1$, and
\begin{align*}
P\left(\left.\bsig_{t+1}=\mathbf{0}\right|\left\{\bsig_t\right\}\right) &= \dfrac{1}{t}\sum_{k=1}^t\left[\left|\bsig_k\right|^2r+\left(1-\left|\bsig_k\right|^2\right)\left(1-\epsilon\right)\right]\\
&= \dfrac{1}{t}\sum_{k=1}^t\left|\boldsymbol{\bsig}_k\right|^2\left(r+\epsilon-1\right)+1-\epsilon.
\end{align*}

The conditional mean of $\bsig_{t+1}$ for $t \geq 1$ is
\begin{align*}
\langle\left.\bsig_{t+1}\right|\left\{\bsig_t\right\}\rangle = &\sum_{\left|\bsig\right|=1}P\left(\left.\bsig_{t+1}=\bsig\right|\left\{\bsig_t\right\}\right)\bsig\\
= &\dfrac{1}{2t}\sum_{k=1}^t\sum_{\left|\bsig\right|=1}\left[\bsig_k\cdot\bsig\gamma+\bsig_k\cdot A\bsig\gamma'+\left(\bsig_k\cdot\bsig\right)^2\left(p+q\right)\right.\\
&\left.+\left(\bsig_k\cdot A\bsig\right)^2\left(p'+q'\right)-\dfrac{1}{2}\left|\bsig_k\right|^2\epsilon\right]\bsig\\
= &\dfrac{1}{2t}\sum_{k=1}^t\sum_{\left|\bsig\right|=1}\left(\bsig_k\cdot\bsig\gamma+\bsig_k\cdot A\bsig\gamma'\right)\bsig\\
= &\dfrac{1}{2t}\sum_{k=1}^t\sum_{\left|\bsig\right|=1}\left(\bsig_k\cdot\bsig\gamma+A\bsig_k\cdot\bsig\gamma'\right)\bsig\\
= &\dfrac{1}{2t}\sum_{k=1}^t2\left(\gamma\bsig_k+\gamma'A\bsig_k\right)\\
= &\dfrac{1}{t}\left(\gamma+\gamma'A\right)\boldsymbol{X}_t.
\end{align*}
Here the symmetry of $\pm\mathbf{i}$, $\pm\mathbf{j}$ is used. Thus,
\begin{align*}
\langle \boldsymbol{X}_{t+1} \rangle = \left(1+\dfrac{\gamma}{t}+\dfrac{\gamma'}{t}A\right)\langle \boldsymbol{X}_t \rangle.
\end{align*}

The conditional mean of $\left|\bsig_{t+1}\right|^2$ for $t \geq 1$ is
\begin{align*}
\langle\left.\left|\bsig_{t+1}\right|^2\right|\left\{\bsig_t\right\}\rangle = &\sum_{\left|\bsig\right|=1}P\left(\left.\bsig_{t+1}=\bsig\right|\left\{\bsig_t\right\}\right)\left|\bsig\right|^2\\
= &\dfrac{1}{2t}\sum_{k=1}^t\sum_{\left|\bsig\right|=1}\left[\bsig_k\cdot\bsig\gamma+\bsig_k\cdot A\bsig\gamma'+\left(\bsig_k\cdot\bsig\right)^2\left(p+q\right)\right.\\
&\left.+\left(\bsig_k\cdot A\bsig\right)^2\left(p'+q'\right)-\dfrac{1}{2}\left|\bsig_k\right|^2\epsilon\right]+\epsilon\\
= &\dfrac{1}{2t}\sum_{k=1}^t2\left(p+q+p'+q'-\epsilon\right)\left|\bsig_k\right|^2+\epsilon\\
= &\dfrac{1-\epsilon-r}{t}\sum_{k=1}^t\left|\bsig_k\right|^2+\epsilon.
\end{align*}
Similar to the 1D case,
\begin{align*}
\langle\left.\left|\bsig_{t+1}\right|^2\right|\left\{\bsig_t\right\}\rangle = \dfrac{t-1}{t}\langle\left.\left|\bsig_t\right|^2\right|\left\{\bsig_{t-1}\right\}\rangle+\dfrac{1-\epsilon-r}{t}\left|\bsig_t\right|^2+\dfrac{\epsilon}{t},
\end{align*}
so
\begin{align*}
\langle\left|\bsig_1\right|^2\rangle &= 1,\\
\langle\left|\bsig_{t+1}\right|^2\rangle &= \left(1-\dfrac{\epsilon+r}{t}\right)\langle\left|\bsig_t\right|^2\rangle+\dfrac{\epsilon}{t}.
\end{align*}
Moreover,
\begin{align*}
\langle\left.\left|\boldsymbol{X}_{t+1}\right|^2\right|\left\{\sigma_t\right\}\rangle &= \left|\boldsymbol{X}_t\right|^2+2\boldsymbol{X}_t\cdot\langle\left.\bsig_{t+1}\right|\left\{\bsig_t\right\}\rangle+\langle\left.\bsig_{t+1}^2\right|\left\{\bsig_t\right\}\rangle\\
&= \left|\boldsymbol{X}_t\right|^2+2\boldsymbol{X}_t\cdot\dfrac{1}{t}\left(\gamma+\gamma'A\right)\boldsymbol{X}_t+\langle\left.\bsig_{t+1}^2\right|\left\{\bsig_t\right\}\rangle\\
&= \left(1+\dfrac{2\gamma}{t}\right)\left|\boldsymbol{X}_t\right|^2+\langle\left.\bsig_{t+1}^2\right|\left\{\bsig_t\right\}\rangle,
\end{align*}
hence
\begin{align*}
\langle \left|\boldsymbol{X}_{t+1}\right|^2\rangle = \left(1+\dfrac{2\gamma}{t}\right)\langle\left|\boldsymbol{X}_t\right|^2\rangle+\langle\bsig_{t+1}^2\rangle.
\end{align*}
By Proposition \ref{ppsigma} and Theorem \ref{thmmoment2},
\begin{align*}
\langle\left|\bsig_t\right|^2\rangle &= C\dfrac{\Gamma\left(t-\epsilon-r\right)}{\Gamma\left(t\right)}+\dfrac{\epsilon}{\epsilon+r},\\
\langle\left|\boldsymbol{X}_t\right|^2\rangle &= \dfrac{\epsilon}{\left(1\!-\!2\gamma\right)\left(\epsilon\!+\!r\right)}t+\dfrac{C}{1\!-\!\epsilon\!-\!r\!-\!2\gamma}\dfrac{\Gamma\left(t\!+\!1\!-\!\epsilon\!-\!r\right)}{\Gamma\left(t\right)}+D\dfrac{\Gamma\left(t\!+\!2\gamma\right)}{\Gamma\left(t\right)},
\end{align*}
where
\begin{align*}
C &= \dfrac{r}{\left(\epsilon+r\right)\Gamma\left(1-\epsilon-r\right)},\\
D &= -\dfrac{1}{\Gamma\left(2\gamma\right)}\left[\dfrac{\epsilon}{\left(\epsilon+r\right)\left(1-2\gamma\right)}+\dfrac{r}{\left(\epsilon+r\right)\left(1-\epsilon-r-2\gamma\right)}\right].
\end{align*}
Due to the above moment formulas, the residual diffusivity results in 1D extend verbatim to the 2D model. 

\section{Conclusions}
We found that residual diffusivity occurs in ERWS models in one and two dimensions
with an inclusion of small probability of symmetric random walk steps. A wedge like sub-diffusive parameter region 
in the $(r,\gamma)$ plane transitions into a diffusive region with residual diffusivity in the sense that 
the enhanced diffusivity strictly exceeds the un-perturbed diffusivity in the limit of vanishing symmetric random walks. 
In future work, we plan to identify other discrete stochastic models for residual diffusivity so that 
the region where this occurs remains distinct from the un-perturbed diffusivity region in the limit of 
vanishing diffusive perturbations.

\end{document}